\documentclass[a4paper,12pt]{article}

\usepackage{color,makeidx,bm,latexsym,amscd,amsmath,amssymb,enumerate,stmaryrd,amsthm,float,graphicx,psfrag,geometry}

\usepackage[all]{xy}

\theoremstyle{plain}
\newtheorem{theorem}{Theorem}[section]

\newtheorem{lemma}[theorem]{Lemma}

\theoremstyle{definition}
\newtheorem{definition}[theorem]{Definition}
\newtheorem{remark}[theorem]{Remark}
\newtheorem{example}[theorem]{Example}

\newcommand{\A}{\mathcal{A}}
\newcommand{\B}{\mathcal{B}}

\newcommand{\F}{\mathbb{F}}
\newcommand{\scF}{\mathcal{F}}
\newcommand{\G}{\mathbb{G}}
\newcommand{\K}{\mathbb{K}}

\newcommand{\Z}{\mathbb{Z}}

\newcommand{\Hom}{\operatorname{Hom}} 
\newcommand{\Ker}{\operatorname{Ker}}

\title{Feynman graphs and Hyperplane arrangements defined over $\mathbb{F}_1$} 
\author{Kyosuke Higashida\thanks{Graduate School of Mathematical Science, The University of Tokyo, 3-8-1 Komaba, Meguroku, Tokyo, 153-8914, JAPAN.
E-mail: k-higashida@g.ecc.u-tokyo.ac.jp}, 
Masahiko Yoshinaga\thanks{(Corresponding author) Department of Mathematics, Faculty of Science, 
Hokkaido University, North 10, West 8, Kita-ku, 
Sapporo 060-0810, JAPAN 
E-mail: yoshinaga@math.sci.hokudai.ac.jp}}
\date{\today}
\begin{document}
\maketitle

\begin{abstract} 
Motivated by some computations of Feynman integrals and certain conjectures 
on mixed Tate motives, Bejleri and Marcolli posed questions about  
the $\mathbb{F}_1$-structure (in the sense of torification) 
on the complement of a hyperplane arrangement, 
especially for an arrangement defined in the space of cycles of a graph. 

In this paper, we prove that an arrangement 
has an $\mathbb{F}_1$-structure 
if and only if 
it is Boolean. We also prove that the arrangement in the cycle space of a 
graph is Boolean if and only if the cycle space has a basis consisting of 
cycles such that any two of them do not share edges. \\
\textbf{Keywords}: Hyperplane arrangements, graphs, torifications. 
\end{abstract}


\section{Introduction}
\label{sec:intro}

Theoretical physics, especially quantum field theory and 
Feynman integrals, raises many mathematical problems \cite{bek}. 
In \cite{bm}, Bejleri and Marcolli studied certain algebraic 
varieties associated with Feynman integrals for graphs 
from the viewpoint of mixed Tate motives and Grothendieck 
ring of varieties. They also discuss the $\mathbb{F}_1$-structure 
(torification, see \cite{lpl}) of these varieties, which is a more recent perspective. 

The category of mixed Tate motives is conjectured to be generated 
by the objects defined by using hyperplane arrangements. 
Therefore, it is a natural question to ask whether a hyperplane 
arrangement has an $\mathbb{F}_1$-structure or not. 
In \cite{bm}, a combinatorial necessary condition 
for an arrangement to have $\F_1$-structure was given. 

The purpose of this paper is to answer questions 
posed in \cite[Question 5.3, 5.4]{bm}. 
Namely, we discuss 
the $\mathbb{F}_1$-structure 
on the complement of a hyperplane arrangement, 
especially for an arrangement defined in the cycle space of a graph 
which is closely related to a variety appearing in the Feynman integral 
\cite[\S 2.2, \S 2.3]{bm}. 

The plan of the present paper is as follows. 
In \S \ref{sec:char}, we recall some basic notions of hyperplane arrangements, 
especially, the notions of the characteristic polynomial and the Boolean arrangement. 
In \S \ref{sec:arr}, 
we show that an arrangement has an $\F_1$-structure if and only if 
it is Boolean. We also show that, under an additional assumption, 
the combinatorial necessary condition formulated in \cite{bm} is sufficient 
for the $\F_1$-structure. 
In \S \ref{sec:graph}, we consider the arrangement defined by a 
graph (\cite[\S 2.3.2]{bm}). We prove that the arrangement is Boolean 
(namely having $\F_1$-structure)  if and only if the first homology group 
of the original graph has a basis consisting of simple cycles which do not 
share edges. 

\section{Characteristic polynomial}
\label{sec:char}

In this section, we recall some basic notions of hyperplane arrangements. 
See \cite{ot} for details. 
Let $\A=\{H_1, \dots, H_n\}$ be an arrangement of affine hyperplanes in 
a vector space $V$ with $\dim V=\ell$ over a field $\K$. The intersection poset is the set 
$L(\A)=\{\cap S\mid S\subset\A \mbox{ with } \cap S\neq\emptyset\}$ of 
nonempty intersections of $\A$. The intersection poset $L(\A)$ is 
partially ordered by reverse inclusion, which has a unique minimal element 
$\widehat{0}=V$. The arrangement $\A$ is said to be \emph{central} if $L(\A)$ has 
a unique maximal element $\cap \A\neq\emptyset$, and is said to be \emph{essential} 
if the maximal elements of $L(\A)$ are $0$-dimensional subspaces. 
An arrangement $\A$ is called a \emph{Boolean arrangement} if $L(\A)$ is 
isomorphic to a Boolean lattice, i.e., the lattice of all subsets of a ground set. 
Let $\A=\{H_1, \dots, H_n\}$ be a Boolean arrangement. 
Then there exists a system of coordinates $(x_1, \dots, x_\ell)$ of $V$ such that 
$H_i$ is equal to the coordinate hyperplane $\{x_i=0\}$ for $1\leq i\leq n$. 
In particular, we have $n\leq \ell$. An arrangement $\A$ is the essential Boolean 
arrangement if and only if $\A$ is essential and $n=\ell$. 

Next, we recall the definition of the characteristic polynomial $\chi(\A, t)$. 
The M\"obius function $\mu:L(\A)\longrightarrow\Z$ is defined by 
\[
\mu(X)=
\begin{cases}
1, & \mbox{ if } X=V, \\
-\sum_{V\leq Y<X}\mu(Y), & \mbox{ if }X>V. 
\end{cases}
\]
Then $\chi(\A, t)=\sum_{X\in L(\A)}\mu(X)t^{\dim X}$. 
One of the most important properties of $\chi(\A, t)$ is the 
following \emph{deletion-restriction formula} (\cite[Cor. 2.57]{ot}). 
Let us fix a hyperplane $H\in\A$. 
Then naturally the deletion $\A':=\A\smallsetminus\{H\}$ and 
restriction $\A'':=H\cap\A'$ are defined. 
Note that $\A''$ is a reduced arrangement in the 
space $H$. Then the following recursive formula holds. 
\[
\chi(\A, t)=\chi(\A', t)-\chi(\A'', t). 
\]
Define the complement of $\A$ by $M(\A)=V\smallsetminus\bigcup_{H\in\A}H$.  
It is easily seen that $M(\A')=M(\A)\sqcup M(\A'')$. Therefore in the 
Grothendieck ring $K_0(\operatorname{\mathcal{V}\mathit{ar}}_{\K})$ of 
varieties over $\K$, we have 
\[
[M(\A)]=\chi(\A, \mathbb{L}), 
\]
where $\mathbb{L}=[\mathbb{A}_{\K}^1]$ is the class of the affine line. 
Suppose $\A$ is defined over $\Z$. Then, for a prime power $q=p^r$ with 
$p\gg 0$, the intersection poset $L(\A)$ is isomorphic to $L(\A\otimes\F_q)$. 
Therefore, $|M(\A\otimes\F_q)|=\chi(\A, q)$. 

\begin{example}
Let $\A$ be the Boolean arrangement defined by $x_1x_2\dots x_n=0$ in 
$\K^\ell$ ($n\leq \ell$). Then $\chi(\A, t)=(t-1)^n t^{\ell-n}$. 
\end{example}

\begin{remark}
Let $\A$ be an arrangement in $\K^\ell$ which is not necessarily 
central. Then the coning $c\A$ (\cite[Definition 1.15]{ot}) is a 
central arrangement in $\K^{\ell+1}$. We have 
$M(c\A)=M(\A)\times\K^\times$. So from now on, we assume that 
all arrangements are central. 
\end{remark}

\section{Arrangements with torified complements}
\label{sec:arr}

A \emph{torification} (\cite{bm, lpl}) of a scheme $X$ over $\Z$ is a morphism of schemes 
$e:T\longrightarrow X$, such that $T=\bigsqcup_j\G_m^{d_j}$ is a disjoint 
union of split tori (where $G_m=\operatorname{Spec}\Z[t, t^{-1}]$), 
the restriction $e|_{\G_m^{d_j}}$ is an isomorphism into 
a locally closed subscheme of $X$, and $e(\K):T(\K)\longrightarrow X(\K)$ 
is bijective for every field $\K$. 

Suppose that 
\begin{itemize}
\item[$(a)$] 
$\A$ is the Boolean arrangement. 
\end{itemize}
Then, there exists a coordinate system such that $\A$ is defined as 
$\{x_1=0\}, \dots, \{x_n=0\}$ in $\K^\ell$ and 
$M(\A)\simeq (\K^\times)^n\times\K^{\ell-n}$. This implies 
(see \cite[\S 1.3.2]{lpl}) 
\begin{itemize}
\item[$(b)$] 
$M(\A)$ has a torification, more precisely, there exists a torified scheme $X$ 
such that $X\otimes\K$ is isomorphic to $M(\A)$ as varieties over $\K$. 
\end{itemize}
Suppose that $M(\A)$ has a torification. Then in the Grothendieck ring, 
$[M(\A)]$ is expressed as a linear combination of $(\mathbb{L}-1)^i$, $i\geq 0$, 
with nonnegative coefficients. Therefore, we have the following. 
\begin{itemize}
\item[$(c)$] 
Suppose $\chi(\A, t)=\sum_i c_i\cdot(t-1)^i$ is the Taylor expansion of 
$\chi(\A, t)$ at $t=1$. Then $c_i\geq 0$ for all $i\geq 0$. 
\end{itemize}
Then, $(a)\Longrightarrow(b)\Longrightarrow(c)$ hold (\cite{bm}). 
Now we discuss other implications. 

\begin{lemma}
\label{lem:coeff}
Let $\A=\{H_1, \dots, H_n\}$ be an arrangement in $V=\K^\ell$. 
If $(c)$ holds, then $n\leq\ell$. 
\end{lemma}
\begin{proof}
Recall that the $\chi(\A, t)$ is a polynomial of  the form 
$t^\ell-n t^{\ell-1}+(\mbox{terms of $\deg\leq \ell-2$})$. 
It is equal to $(t-1)^\ell+(\ell-n)(t-1)^{\ell-1}+(\mbox{terms of $\deg\leq \ell-2$})$. 
Hence $(c)$ implies $\ell-n\geq 0$. 
\end{proof}

\begin{theorem}
\label{thm:main1}
Let $\A=\{H_1, \dots, H_n\}$ be an arrangement in $\K^\ell$. Then, 
\begin{itemize}
\item[(1)] $(a)$ and $(b)$ are equivalent. 
\item[(2)] Suppose $\A$ is essential. Then $(a)$, $(b)$ and $(c)$ are equivalent. 
\end{itemize}
\end{theorem}
\begin{proof}
We first consider (2). 
If $\A$ is essential, then by definition, $n\geq \ell$. Suppose $(c)$ holds. 
Then Lemma \ref{lem:coeff} tells that $n\leq \ell$, hence $n=\ell$, which 
implies $(a)$. 

Next we consider (1). We need to prove $(b)\Longrightarrow (a)$. 
Let us assume that $M(\A)$ has a torification. Then there exists a 
dominant morphism $(\K^\times)^\ell\hookrightarrow M(\A)$. 
Suppose that $\A$ is not Boolean. 
Then there exist dependent hyperplanes. Namely, after a change of 
coordinates, we may assume $\{x_1=0\}, \dots, \{x_r=0\}, 
\{x_1+x_2+\dots+x_r=0\}$ (with $2\leq r\leq n-1$) are in $\A$. 
Then we have $(\K^\times)^\ell\hookrightarrow M(\A)\subset 
\{x_1x_2\dots x_r(x_1+\dots+x_r)\neq 0\}\subset\K^\ell$. 
Taking the ring of functions, we have an embedding of $\K$-algebras. 
\[
\K\left[x_1^{\pm 1}, \dots, x_r^{\pm 1}, x_{r+1}, \dots, x_\ell, \frac{1}{x_1+\dots+x_r}
\right]
\subset
\K[t_1^{\pm 1}, \dots, t_\ell^{\pm 1}]. 
\]
However, this is impossible. Because in the right hand side, the set of 
invertible elements is equal to the set of monomials, which are linearly 
independent over $\K$. However, in the left hand side, 
$x_1, \dots, x_r, x_1+\dots+x_r$ are dependent invertible elements. 
\end{proof}

The following example shows that if $\A$ is not essential, 
$(c)$ is not sufficient for $(a)$ or $(b)$. 
\begin{example}
Consider the arrangement $\A$ defined by $x_1x_2x_3(x_1+x_2+x_3)=0$ in $\K^4$. 
Then, $\chi(\A, t)=t^4-4t^3+6t^2-3t=(t-1)^4+(t-1)$. Hence the coefficients 
of $(t-1)^i$ are nonnegative, however, $\A$ is not Boolean.  
\end{example}

\section{The arrangement in the cycle space of a graph}
\label{sec:graph}

Let $\Gamma=(V, E)$ be a finite graph. Let $H_1(\Gamma)=H_1(\Gamma, \K)$ 
be the cycle space over $\K$. 

Since $\Gamma$ is a $1$-dimensional CW-complex, 
any $1$-cochain is automatically a cocycle. 
Therefore, any edge $e\in E$ (equipped with an orientation) determines an element 
of $H^1(\Gamma)=\Hom(H_1(\Gamma), \K)$. We denote the element by 
$\eta_e:H_1(\Gamma)\longrightarrow\K$. 

If $\eta_e\neq 0$, then $H_e:=\Ker\eta_e$ defines a hyperplane in 
$H_1(\Gamma)\simeq\K^{b_1(\Gamma)}$. 
As in \cite{bm}, define $\A_\Gamma$ as 
$\A_\Gamma:=\{H_e\mid e\in E, \eta_e\neq 0\}$. 
We note that $e, e'\in E$ with $e\neq e'$ may define the same hyperplane. 
We just forget the multiplicities and consider the reduced arrangement. 

\begin{lemma}
\label{lem:essential}
The arrangement $\A_\Gamma$ is essential. 
In particular, $|\A_\Gamma|\geq b_1(\Gamma)$. 
\end{lemma}
\begin{proof}
Since $\{e\in E\}$ spans the space of $1$-cochains, $\{\eta_e\}$ generates 
the space of $1$-cocycles, hence $H^1(\Gamma)$. We have $\cap_{e\in E}H_e=\{0\}$. 
\end{proof}

Let $\scF$ be a spanning forest of $\Gamma$. 
Let $e\in E\smallsetminus\scF$. Then the vertices of $e$ are connected by 
the unique minimal path in $\scF$. By adding $e$ to this path, we have a 
simple cycle $C_e\subset E$. 
By putting suitable orientations, the set of such cycles 
$\B(\scF):=\{C_e\subset E\mid e\in E\smallsetminus \scF\}$ 
forms a basis of $H_1(\Gamma)$. 
Such a basis is called a fundamental basis \cite{sys}. 

\begin{definition}
Let $\B\subset H_1(\Gamma)$ be a basis consisting of simple cycles (that is, 
each cycle passes an edge at most once). 
Then $\B$ is said to be separated if any two cycles $C, C'\in\B$ ($C\neq C'$) 
do not share edges. 
\end{definition}

\begin{theorem}
\label{thm:main2}
Let $\Gamma=(V, E)$ be a graph. Then the following are equivalent. 
\begin{itemize}
\item[$(i)$] $M(\A_\Gamma)$ has a torification. 
\item[$(ii)$] $\A_\Gamma$ is Boolean. 
\item[$(iii)$] $|\A_\Gamma|=b_1(\Gamma)$. 
\item[$(iv)$] 
For any spanning forest $\scF\subset\Gamma$, the fundamental basis 
$\B(\scF)$ is separated. 
\item[$(v)$] 
There exists a spanning forest $\scF\subset\Gamma$ such that the fundamental basis 
$\B(\scF)$ is separated. 
\item[$(vi)$] 
There exists a separated basis $\B$. 
\end{itemize}
\end{theorem}
\begin{proof}
$(i)\Longleftrightarrow (ii)\Longleftrightarrow (iii)$ is obtained from 
Theorem \ref{thm:main1} and Lemma \ref{lem:essential}. 
Also $(iv)\Longrightarrow (v)\Longrightarrow (vi)$ is obvious. 
Now we assume $(vi)$. 
Let $\B=\{C_1, \dots, C_b\}$ be a basis as in $(vi)$. Then every 
spanning forest is obtained from $\Gamma$ by removing an edge of $C_i$ for each $i$. 
Thus $\B$ is the fundamental basis for every spanning forest $\scF$. 
This proves $(vi)\Longrightarrow (iv)$. 

Now we prove $(vi)\Longrightarrow (ii)$. Let $\scF$ be a basis as in $(vi)$. 
Then every edge $e\in  E$ is either contained in the unique cycle $C\in\B$ or 
not contained in $\bigcup_{C\in\B}C$. In the latter case, $\eta_e=0$, hence we 
do not need to consider. Hence $\{\eta_e\}\subset H^1(\Gamma)$ 
forms a dual basis of $\B$. Thus $\A_\Gamma$ is Boolean. 

Finally, we prove $(ii)\Longrightarrow (iv)$. Suppose $(ii)$ and there exists 
a spanning forest $\scF$ such that $\B(\scF)$ is not separated. There exists 
an edges $e_0\in E$ such that $\{C\in\B(\scF)\mid C\ni e_0\}$ contains 
more than one cycles. 
Let $E\smallsetminus \scF=\{e_1, \dots, e_b\}$. Then 
$\{H_{e_0}, H_{e_1}, \dots, H_{e_b}\}\subset\A_\Gamma$ forms $b+1$ distinct hyperplanes. 
This contradicts the fact that $\A_\Gamma$ is Boolean. This completes 
the proof. 
\end{proof}

\begin{remark}
The graph satisfying the conditions in Theorem \ref{thm:main2} is 
obtained by finitely many repetitions of gluing a graph with $b_1\leq 1$ at a vertex. 
\end{remark}

\begin{remark}
\label{rem:revision}
There is another way to associate a hyperplane arrangement $\A(\Gamma)$ 
to a graph $\Gamma$, which is called the \emph{graphic arrangement} 
(see \cite[\S 2.4]{ot} for details). The two arrangements $\A_\Gamma$ and 
$\A(\Gamma)$ are dual to each other in the sense of matroids (see \cite{ox}). 
In view of this interpretation, Theorem \ref{thm:main2} characterizes those 
graphs $\Gamma$ for which the simplification of the dual of the circuit matroid 
of $\Gamma$ is the free matroid. 
\end{remark}

\medskip

\noindent
{\bf Acknowledgements.} 
Masahiko Yoshinaga 
was partially supported by JSPS KAKENHI 
Grant Numbers JP19K21826, JP18H01115. 
The authors thank the referee for pointing out Remark \ref{rem:revision}.

\end{document}